\documentclass[a4paper,12pt]{article}
\usepackage[english]{babel}
\usepackage[utf8]{inputenc}
\usepackage{amssymb}
\usepackage{amsthm}
\usepackage{amsmath}
\usepackage{cite}

\usepackage[nottoc]{tocbibind}

\newtheorem{theorem}{Theorem}[section]
\newtheorem{corollary}[theorem]{Corollary}
\newtheorem{definition}[theorem]{Definition}
\newtheorem{lemma}[theorem]{Lemma}

\title{Asymptotics for cliques in scale-free random graphs}
\author{Fraser Daly\footnote{Department of Actuarial Mathematics and Statistics, Heriot-Watt University, Edinburgh EH14 4AS, UK.  E-mail: f.daly@hw.ac.uk}, Alastair Haig\footnote{Department of Actuarial Mathematics and Statistics, Heriot-Watt University, Edinburgh EH14 4AS, UK.  E-mail: ahh6@hw.ac.uk} and Seva Shneer\footnote{Department of Actuarial Mathematics and Statistics, Heriot-Watt University, Edinburgh EH14 4AS, UK.  E-mail: v.shneer@hw.ac.uk} }
\date{\today}

\begin{document}

\maketitle

\noindent{\bf Abstract} 
In this paper we establish asymptotics (as the size of the graph grows to infinity) for the expected number of cliques in the Chung--Lu inhomogeneous random graph model in which vertices are assigned independent weights which have tail probabilities $h^{1-\alpha}l(h)$, where $\alpha>2$ and $l$ is a slowly varying function. Each pair of vertices is connected by an edge with a probability proportional to the product of the weights of those vertices. We present a complete set of asymptotics for all clique sizes and for all non-integer $\alpha > 2$. We also explain why the case of an integer $\alpha$ is different, and present partial results for the asymptotics in that case.
\vspace{12pt}

\noindent{\bf Key words and phrases:} Chung--Lu model; inhomogeneous random graph; clique; slowly varying function

\vspace{12pt}

\noindent{\bf MSC 2020 subject classification:} 05C80; 60F05

\section{Introduction}

Scale-free networks are ubiquitous in the modern world. In such networks, the number of nodes with degree $k$ decays slowly for large $k$, so that there are a few nodes with extremely large degrees, even when the average degree is relatively small. In many cases (including the model we consider here) the proportion of nodes with specific degrees behaves similarly to $k^{-\alpha}$ for some exponent $\alpha$. The Internet, the IMDB movie collaboration network and even some semantic networks are all said to be examples of such networks \cite{hofstad_2016}.

Recall that a clique in a graph is a subset of vertices that form a complete subgraph. In this paper we study how, in a scale-free graph, the expected number of cliques of a given size varies asymptotically as $n$ (the number of nodes) grows, as well as the intrinsically linked probability of a given set of nodes comprising a clique.
Janssen, van Leeuwaarden and Shneer \cite{Janssen_2019} have studied this for the power-law exponent in the interval $(2,3)$. Here we extend those findings to exponents in $(2,\infty)$ to find that the asymptotics not only depend on the clique size $k$ and the power-law exponent $\alpha$, but also their relation to each other and, in some cases, whether $\alpha$ is an integer. This provides a baseline, or null, growth rate against which to compare the expected number of cliques in given data or models as a potential measure of connectedness or of community or anti-community structure.

We use the rank-1 inhomogeneous graph model (also known as the hidden-variable model) \cite{PhysRevE.68.036112,Bollob_s_2007,Britton_2006,Chung15879,norros_reittu_2006,PhysRevE.70.066117} or, more specifically, the Chung--Lu variant of the rank-1 inhomogeneous graph model \cite{Chung15879} with power-law exponent $\alpha$ in the range $(2,\infty)$. In this model (which is fully described in Section \ref{theModel}) each node is assigned a weight from a distribution with close to power-law tails and, conditional on these weights, edges exist independently of each other in such a way that the expected degree of each node is close to its weight \cite{PhysRevE.68.036112}. For large enough $n$ this results in the expected degree sequence of the graph matching the close to power-law pattern observed in real world scale-free networks \cite{Voitalov_2019}. Further studies into this model have also revealed that degree correlations, average connectivity and clustering coefficients match the observed values of these properties in various real life networks \cite{stegehuis2017degree,PhysRevE.96.042309}. 

\subsection{Relation to existing work}

One of the primary difficulties encountered in using this model is the minimum function present in the expression for edge probabilities; see (\ref{edgeProb}) in Section \ref{theModel} below, where the average weight parameter $\mu$ of the model is also introduced. A way to circumvent this is to truncate the support of the weight distribution at $\sqrt{\mu n}$, as opposed to the infinite support we use. In \cite{Bianconi_2006}, Bianconi and Marsili address this case and obtain asymptotics for the clique counts. Notably this cutoff case is identical to one of the `extreme cases' described in Section \ref{sec:main} below.

As mentioned above, \cite{Janssen_2019} sees Janssen, van Leeuwaarden and Shneer obtain sharp asymptotics for the weight distribution with infinite support, but restrict their view to power-law exponents in $(2,3)$. Some of the methodologies used here for power-law exponents in $(2,\infty)$ are generalisations of those seen in that paper.

An alternative line of work is considered by Van der Hofstad \textit{et al.} \cite{hofstad2017optimal}, who investigate the optimal composition for the most likely subgraph. Notably, they showed that for many subgraphs (including cliques) the optimal composition contains entirely nodes with degree of order $\sqrt{n}$ and that, up to leading order, said compositions determine how the expected number of copies of a given subgraph varies as $n$ increases. Our results are consistent with these findings, as we see that in the case $k > \alpha$ the asymptotics correspond to nodes with weights near the boundary $\sqrt{\mu n}$.

\subsection{Outline of this paper}

In Section \ref{theModel}, we formally introduce the model we study and some notation we will use throughout the subsequent work. Our main results are stated and proved in Section \ref{sec:main}, with proofs of some lemmas deferred until Appendix \ref{InitProofs}. The main results of this paper apply in the case where $\alpha$ is not an integer. Some remarks on this and an investigation of the case of integer $\alpha$ are given in Section \ref{alphaIntSec}.

\section{Model and notation} \label{theModel}

We will be using the Chung--Lu version of the rank-1 inhomogeneous graph model \cite{Chung15879} with $n$ nodes. Node $i$ is assigned weight $H_i$, where $(H_1, H_2, \ldots, H_n)$ is as an i.i.d. sample from a random variable $H$ with tail distribution
\begin{align}
    \overline{F}(h) = \mathbb{P}\left( H > h \right) = h^{1-\alpha} l(h), \quad h \geq 1\,,
\end{align}
for some value $\alpha \in \left(2, \infty \right)$ and $l(h)$ which is slowly-varying. Here we define slowly-varying functions in the Karamata sense \cite{bingham_goldie_teugels_1987}, that is, $l$ is slowly varying if
\begin{align}
   \lim_{x\rightarrow \infty} \frac{l(\lambda x)}{l(x)} = 1\,,
\end{align}
for all $\lambda>0$.

A pair of nodes $(i,j)$ with weights $(h_i,h_j)$, conditional on their weights being $h_i$ and $h_j$, is connected by an edge, independently of everything else, with probability
\begin{align} \label{edgeProb}
    p_{i,j} = \min \left\{ \frac{h_ih_j}{\mu n}, 1\right\}\,,
\end{align}
where $\mu = \mathbb{E}(H)$ is a parameter of the model, thought of as an average weight.

We use this model due to its well-known property that the expected degree of a given node, conditional on its own weight, is close to that weight. Hence, we are able to see that for large enough graphs and large enough degrees, the expected degree sequence resembles  $h^{-\alpha}l(h)$ -- the `close-to power-law' distribution we would expect from a scale-free graph.

\subsection{Notation}

We list here some notation that we will use throughout the work that follows. For functions $f,g:\mathbb{R}\to\mathbb{R}$,
\begin{enumerate}
\item we write $f(n) \sim g(n)$ if $\frac{f(n)}{g(n)} \rightarrow 1$ as $n \rightarrow \infty$.
\item we write $f(n) \asymp g(n)$ if there exist constants $n_0 \geq 0$ and $0 < C_1 < C_2 < \infty$ such that, for all $n > n_0$,
    \begin{align*}
        C_1g(n) \leq f(n) \leq C_2g(n)\,.
    \end{align*}
Note that if $f(n) \sim Bg(n)$ for some positive constant $B$, then $f(n) \asymp g(n)$.
\item we write $f(n) \lesssim g(n)$ if there exists a constant $n_0 \geq 0$ and a function $h(n)$ such that, for all $n > n_0$, $f(n) \leq h(n)$ and $h(n) \asymp g(n)$.
\end{enumerate}

\section{Main results} \label{sec:main}

In this section we state and prove our main results, asymptotics for the expected number of cliques in the random graph model defined in Section \ref{theModel} above. Let $A_k(n)$ denote the expected number of cliques of size $k$ in this model. We are interested in how, for fixed $\alpha$, $\mu$ and $k$, and a fixed, slowly-varying function $l$, $A_k(n)$ varies asymptotically with $n$. 

Letting $S_k$ denote the set of all subsets of $\left\{1,2,\ldots, n \right\}$ of size $k$, $K_{\mathbf{s}}$ denote the event that nodes $\mathbf{s} = \left\{ s_1,s_2,\ldots,s_k \right\}$ form a clique, and $K_k$ denote the event that the specific nodes $\left\{ 1,2,\ldots k \right\}$ form a clique, we note that
\begin{equation}
A_k(n) = \mathbb{E} \left[ \sum_{\mathbf{s} \in S_k} \mathbb{I}\left( K_{\mathbf{s}}\right)\right]
= \sum_{\mathbf{s} \in S_k} \mathbb{P} \left( K_{\mathbf{s}}\right)
= {n \choose k} \mathbb{P} \left( K_{k}\right) 
\asymp n^k  \mathbb{P} \left( K_{k}\right)\,, \label{AkPKkLink}
\end{equation}
where $\mathbb{I}\left( A \right)$ is an indicator function for the event $A$. We therefore focus primarily on how $\mathbb{P}\left( K_k \right)$ varies asymptotically with increasing $n$.

Our main result is Theorem \ref{alphaNonInt} below. This gives sharp asymptotics for $\mathbb{P}(K_k)$ in the case where $\alpha$ is not an integer; we will discuss the case of integer $\alpha$ in Section \ref{alphaIntSec}. 

\begin{theorem} \label{alphaNonInt}
For $\alpha \in (2,\infty) \setminus \mathbb{Z}$ and $k \geq 2$,
\begin{align}
    \mathbb{P} \left( K_k \right) \asymp \begin{cases} n^{\frac{k}{2}(1-k)} & \text{ when } k<\alpha\,. \\
    n^{\frac{k}{2}(1-\alpha)} l \left( \sqrt{n} \right)^{k} & \text{ when } k>\alpha\,.
    \end{cases}
\end{align}
\end{theorem}

The following corollary is immediate from (\ref{AkPKkLink}).

\begin{corollary} \label{AkalphaNonInt}
For $\alpha \in (2,\infty) \setminus \mathbb{Z}$ and $k \geq 2$,
\begin{align}
    A_k(n) \asymp \begin{cases} n^{\frac{k}{2}(3-k)} & \text{ when } k<\alpha\,, \\
    n^{\frac{k}{2}(3-\alpha)} l \left( \sqrt{n} \right)^{k} & \text{ when } k>\alpha\,.
    \end{cases}
\end{align}
\end{corollary}

Notably, we see here that if $\alpha > 3$ then the average number of cliques of size larger than $3$ decreases with $n$. That is, the graph becomes more and more sparse as it increases in size. 

We use the remainder of this section to prove Theorem \ref{alphaNonInt}. Recalling (\ref{edgeProb}), note that if $h_i,h_j \leq \sqrt{\mu n}$, then $p_{ij}= \frac{h_ih_j}{\mu n}$, and if $h_i,h_j > \sqrt{\mu n}$ then $p_{ij} = 1$, so we may condition on whether the weights of certain nodes are less than or greater than $\sqrt{\mu n}$. At this stage nodes are interchangeable, so we may condition only on the number of nodes with weights less than or greater than $\sqrt{\mu n}$. Hence, we can write
\begin{equation}\label{eq:cond}
    \mathbb{P} \left( K_k \right) = \mathbb{P} \left( K_k, H_i \leq \sqrt{\mu n}, 1 \leq i \leq k \right) + \sum_{m=1}^{k-1} {k \choose m} I_m +  \mathbb{P} \left( K_k, H_j > \sqrt{\mu n}, 1 \leq j \leq k \right)\,,
\end{equation}
where $I_m = \mathbb{P}\left( K_k, H_i \leq \sqrt{\mu n}, H_j > \sqrt{\mu n}, 1 \leq i \leq m, m < j \leq k \right)$. We refer to the first and last terms on the right-hand side of (\ref{eq:cond}) as the `extreme cases' and the remaining terms as the `intermediate cases', as in \cite{Janssen_2019}.

Consider the extreme cases, starting with the case where every node has weight greater than $\sqrt{\mu n}$. Since, in this case, every edge exists with probability $1$, we may disregard them altogether to write
\begin{multline}
    \mathbb{P} \left( K_k, H_j > \sqrt{\mu n}, 1 \leq j \leq k \right) = \mathbb{P} \left( H_j > \sqrt{\mu n}, 1 \leq j \leq k \right)
   = \mathbb{P} \left( H_1 > \sqrt{\mu n} \right)^k\\
      = \overline{F}(\sqrt{\mu n})^k
    = \mu^{\frac{k}{2}(1-\alpha)}n^{\frac{k}{2}(1-\alpha)}l(\sqrt{\mu n})^k 
    \asymp n^{\frac{k}{2}(1-\alpha)}l(\sqrt{n})^k\,. \label{extremeCase1}
\end{multline}
For the other extreme case, we are able to calculate the asymptotics directly:
\begin{align}
    \mathbb{P} \left( K_k, H_i \leq \sqrt{\mu n}, 1 \leq i \leq k \right) &= \int_1^{\sqrt{\mu n}} \ldots \int_1^{\sqrt{\mu n}} \prod_{1\leq i < j \leq k} \frac{h_ih_j}{\mu n}\ dF(h_k)\ldots dF(h_1) \nonumber \\
    &= \mu^{\frac{k}{2}(1-k)} n^{\frac{k}{2}(1-k)} \int_1^{\sqrt{\mu n}} \ldots \int_1^{\sqrt{\mu n}} \prod_{i=1}^k h_i^{k-1} \ dF(h_k)\ldots dF(h_1) \nonumber \\
    &\asymp n^{\frac{k}{2}(1-k)} \left( \int_1^{\sqrt{\mu n}} h_1^{k-1}\ dF(h_1)\right)^k \nonumber \\
    & \asymp \begin{cases} n^{\frac{k}{2}(1-k)} &  \text{ when } k < \alpha\,, \\ n^{\frac{k}{2}(1-\alpha)} l(\sqrt{n})^k & \text{ when } k > \alpha\,.
    \end{cases} \label{extremeCase2}
\end{align}
In the last step above we made use of Lemma \ref{fixedAsymps} below. Note that if $k<\alpha$ then $n^{\frac{k}{2}(1-k)}$ dominates $n^{\frac{k}{2}(1-\alpha)}l(\sqrt{n})^k$, so this extreme case asymptotically dominates that considered in (\ref{extremeCase1}).

From these arguments, we see that the asymptotics presented in Theorem \ref{alphaNonInt} are at least a lower bound for the asymptotics of $\mathbb{P}\left( K_k \right)$. In Section \ref{alphaNonIntProof} below, we prove that when $\alpha$ is non-integer, each of the intermediate cases are also bounded above by these asymptotics, so they are also an asymptotic upper bound (and hence a sharp asymptotic) for $\mathbb{P} \left( K_k \right)$.

\subsection{Preliminary lemmas}\label{sec:lemmas}

Before we consider upper bounds for the intermediate cases, and thus complete the proof of Theorem \ref{alphaNonInt}, we first state some useful results, the proofs of which can be found in Appendix \ref{InitProofs}.

\begin{lemma} \label{fixedAsymps}
Let $\alpha,\beta,X$ be real constants, $l(h)$ be a slowly varying function which is locally bounded for $h>h_0$, and $x(n)$ and $y(n)$ be non-decreasing functions such that, for all $n\geq1$,
\begin{align*}
    h_0 \leq X \leq x(n) \leq y(n) < \infty  
    \text{ and } \lim_{n \rightarrow \infty} x(n) = \infty\,.
\end{align*}
Let $F(h) = 1 - \overline{F}(h) = 1 - h^{1-\alpha}l(h)$. Then
\begin{align}
    \int_X^{x(n)} h^\beta\ dF(h) &\asymp \begin{cases} x(n)^{\beta-\alpha+1} l(x(n)) & \text{ when } \beta -\alpha > -1\,, \\
1 & \text{ when } \beta -\alpha < -1\,, 
\end{cases}  \label{fixedBottom} \\
\int_{x(n)}^{y(n)} h^\beta\ dF(h) &\asymp \begin{cases} y(n)^{\beta-\alpha+1} l(y(n)) & \text{ when } \beta -\alpha > -1\,, \\
x(n)^{\beta-\alpha+1} l(x(n)) & \text{ when } \beta -\alpha < -1\,. 
\end{cases}  \label{fixedNone}
\end{align}
\end{lemma}

\begin{definition}
For integers $i$ and $m$ such that $1 \leq i \leq m$, we define the linear functional $J_{i,m}$ from the set of functions on $(n,h_1,h_2,\ldots,h_{i-1},h_i)$ to the set of functions on $(n,h_1,h_2,\ldots,h_{i-1})$ as follows:
\begin{align}
    J_{i,m} \left( g\right)(n,h_1,h_2,\ldots,h_{i-1}) = \int_1^{h_{i-1}} h_i^{m-1} g(n,h_1,h_2,\ldots,h_{i-1},h_i)\ dF(h_i)\,,
\end{align}
where we use the convention that $h_0 = \sqrt{\mu n}$.  We will often write $J_{i,m}(g)$ instead of $J_{i,m}(g)(n,h_1,h_2,\ldots,h_{i-1})$ where it does not create confusion, in order to simplify notation.

We also define linear functional $\overline{J}_{i,m}$ from the set of functions on $(n,h_1,h_2,\ldots,h_m)$ to the set of functions on $(n,h_1,h_2,\ldots,h_{i-1})$ as
\begin{align}
    \overline{J}_{i,m} \left( g \right)(n,h_1,h_2,\ldots,h_{i-1}) = J_{i,m} \left( J_{i+1,m} \left( \ldots J_{m,m} \left( g \right)\ldots \right) \right)\,.
\end{align}
\end{definition}

\begin{corollary} \label{JCor}
If
\begin{align*}
    g(n,h_1,h_2,\ldots,h_{i-1},h_i) \begin{Bmatrix} \asymp \\ \text{or} \\ \lesssim \end{Bmatrix} B(n,h_1,h_2,\ldots,h_{i-1})h_i^\beta\,,
\end{align*}
for some positive function $B$ and $\beta \neq \alpha - 1$, then
\begin{align*}
    J_{i,m} (g) \begin{Bmatrix} \asymp \\ \text{or} \\ \lesssim \end{Bmatrix} \begin{cases} B(n,h_1,h_2,\ldots,h_{i-1}) h_{i-1}^{\beta+m-\alpha}l(h_{i-1}) & \text{ when } \beta + m - \alpha > 0\,, \\
        B(n,h_1,h_2,\ldots,h_{i-1}) & \text{ when } \beta + m - \alpha < 0\,.
    \end{cases}
\end{align*}
\end{corollary}

\begin{lemma} \label{escapeLemma}
For constant $\gamma > 0$, slowly varying function $l(h)$ and $h_i \leq \sqrt{\mu n}$
\begin{align}
    h_i^{\gamma} l \left( h_i \right) \lesssim n^{\frac{\gamma}{2}} l \left( \sqrt{n} \right) \qquad \text{ and } \qquad \frac{n^{-\gamma}}{h_i^{-\gamma}}l \left( \frac{\mu n}{h_i} \right) \lesssim n^{-\frac{\gamma}{2}} l \left( \sqrt{n} \right)\,.
\end{align}
\end{lemma}

\begin{lemma} \label{powerDominatesSlowVary}
If $l(x)$ is a slowly-varying function, then, for any $a\in\mathbb{R}$ and $\epsilon > 0$,
\begin{align*}
    x^{a-\epsilon} l (x) \lesssim x^a\,.
\end{align*}
\end{lemma}

\begin{lemma} \label{convexLemma}
If $x_1,x_2,\ldots,x_m \geq 0$ and $v\geq 1$, then
\begin{align*}
    (x_1+x_2+\ldots+x_m)^v \lesssim x_1^v+x_2^v+\ldots+x_m^v\,.
\end{align*}
\end{lemma}

\subsection{Upper bounds for Theorem \ref{alphaNonInt}} \label{alphaNonIntProof}

To show that the intermediate terms are all asymptotically bounded above by the relevant asymptotics we subdivide the $I_m$ defined below (\ref{eq:cond}) into two further cases: $m > \alpha$ and $m \leq \alpha$. 

In the case where $m > \alpha$ (which only occurs when $k > \alpha$) we follow the methodology in \cite{Janssen_2019} and consider the subgraph of the $k$-clique consisting of an $m$-clique over the nodes whose weights are less than or equal to $\sqrt{\mu n}$ and a $(k-m)$-clique over the nodes whose weights are greater than $\sqrt{\mu n}$. Notably the two cliques in the subgraph are disjoint and independent. Hence,
\begin{align*}
     I_m &= \mathbb{P}\left( K_k, H_i \leq \sqrt{\mu n}, H_j > \sqrt{n}, 1 \leq i \leq m, m < j \leq k \right) \\
     &\leq \mathbb{P}\left( K_m, H_i \leq \sqrt{\mu n}, 1 \leq i \leq m \left) \mathbb{P}\right(K_{k-m}' , H_j' > \sqrt{\mu n}, 1 \leq j \leq k-m \right) \\
     &\asymp n^{\frac{m}{2}(1-\alpha)}l(\sqrt{n})^{m}n^{\frac{k-m}{2}(1-\alpha)} l(\sqrt{n})^{k-m} \\
     &= n^{\frac{k}{2}(1-\alpha)}l(\sqrt{n})^k\,,
\end{align*}
where, in the penultimate step, we used the results for the extreme cases established above. Hence, if $m>\alpha$ then $I_m \lesssim n^{\frac{k}{2}(1-\alpha)}l(\sqrt{n})^k$.

We now consider the case $m \leq \alpha$ and start by making two crucial observations. First, the nodes with weights conditioned to be less than $\sqrt{\mu n}$ are interchangeable, so without loss of generality we can say that $1 \leq H_m \leq H_{m-1} \leq \ldots \leq H_1 \leq \sqrt{\mu n}$. Second, since the remaining nodes all have weights greater that $\sqrt{\mu n}$ they are connected to each other with probability $1$ and connected to the nodes of weight less than or equal to $\sqrt{\mu n}$ (conditional on the weights of said nodes) independently of each other. Hence we can consider these nodes to be $(k-m)$ independent copies of each other. So, letting $v=k-m$ for ease of notation, we can write
\begin{align*}
    I_m 
    &= m! \mathbb{P}\left( K_k, H_m \leq H_{m-1} \leq \ldots \leq H_1 \leq \sqrt{\mu n}, H_j > \sqrt{\mu n}, m < j \leq k \right) \\
    &= m! \int_1^{\sqrt{\mu n}} \int_1^{h_1} \ldots \int_1^{h_{m-1}} \left(\int_{\sqrt{\mu n}}^{\infty} \ldots \int_{\sqrt{\mu n}}^{\infty} \prod_{1\leq i < j \leq m+v} p_{i,j}\; dF(h_{m+v}) \ldots dF(h_{m+1}) \right) \\
    & \qquad \qquad \qquad \qquad \qquad \qquad \qquad \qquad \qquad \qquad \qquad \ \  dF(h_m) \ldots dF(h_2)\; dF(h_1) \\
    &= m! \int_1^{\sqrt{\mu n}} \int_1^{h_1} \ldots \int_1^{h_{m-1}} \prod_{1\leq i < j \leq m} \frac{h_ih_j}{\mu n} \left(\int_{\sqrt{\mu n}}^{\infty} \prod_{i=1}^m p_{i,m+1}\; dF(h_{m+1}) \right)^v \\
    & \qquad \qquad \qquad \qquad \qquad \qquad \qquad \qquad \qquad \ \ dF(h_m) \ldots dF(h_2)\; dF(h_1) \\
    &\asymp n^{\frac{m}{2}(1-m)} \int_1^{\sqrt{\mu n}} h_1^{m-1} \int_1^{h_1} h_2^{m-1} \ldots \int_1^{h_{m-1}} h_m^{m-1} \left(\int_{\sqrt{\mu n}}^{\infty} \prod_{i=1}^m p_{i,m+1}\; dF(h_{m+1}) \right)^v \\
    & \qquad \qquad \qquad \qquad \qquad \qquad  \qquad \qquad \qquad \qquad \qquad \qquad dF(h_m) \ldots dF(h_2)\; dF(h_1) \\
    &= n^{\frac{m}{2}(1-m)} \overline{J}_{1,m} \left( \left( \int_{\sqrt{\mu n}}^{\infty} \prod_{i=1}^m p_{i,j}\; dF(h_{j}) \right)^v \right)\,.
\end{align*}
Hence, in order to show that $I_m$ is asymptotically less than or equal to the desired function, we only need to show that
\begin{equation*}
\overline{J}_{1,m} \left( \left( \int_{\sqrt{\mu n}}^{\infty} \prod_{i=1}^m p_{i,j}\; dF(h_{j}) \right)^v \right) \lesssim \begin{cases} n^{\frac{k}{2}(1-k) - \frac{m}{2}(1-m)} & \text{ when } k<\alpha\,, \\
    n^{\frac{k}{2}(1-\alpha)-\frac{m}{2}(1-m)} l \left( \sqrt{n} \right)^{k} & \text{ when } k>\alpha\,.
    \end{cases}
\end{equation*}
We will use the following lemma.

\begin{lemma}\label{middleExpansionLem}
\begin{multline}
    \int_{\sqrt{\mu n}}^{\infty} \prod_{i=1}^m p_{i,j}\; dF(h_{j})\\ \asymp C_0  n^{\frac{1}{2}(1-m-\alpha)} l(\sqrt{\mu n}) \prod_{i=1}^{m} h_i + \sum_{s=1}^{m} C_s n^{1-\alpha} h_s^{-(m-s+1-\alpha)} l \left( \frac{\mu n}{h_s} \right) \prod_{i=s+1}^{m} h_i\,, \label{middleExpansion}
\end{multline}
for some non-negative constants $C_0, C_1, \ldots, C_m$.
\end{lemma}

\begin{proof}
Noting that if $h_j > \frac{\mu n}{h_i}$ then $p_{i,j} = 1$, we see (using the convention that $h_0 = \sqrt{\mu n}$),
\begin{align} \label{middleExpansionIntegrals}
    \int_{\sqrt{\mu n}}^{\infty} \prod_{i=1}^m p_{i,j}\; dF(h_{j}) =& \sum_{s=1}^{m} \int_{\frac{\mu n}{h_{s-1}}}^{\frac{\mu n}{h_s}} \prod_{i=s}^{m} \frac{h_ih_j}{\mu n} \; dF(h_j) + \int_{\frac{\mu n}{h_m}}^\infty \; dF(h_j) \nonumber \\
    \asymp& \sum_{s=1}^{m} \mu_{s} n^{-(m-s+1)} \prod_{i=s}^{m} h_i \int_{\frac{\mu n}{h_{s-1}}}^{\frac{\mu n}{h_s}} h_j^{m-s+1} \; dF(h_j) + \overline{F}\left( \frac{\mu n}{h_m} \right)\,,
\end{align}
where $\mu_s = \mu^{-(m-s+1)}$. We proceed to show that each of the terms here is asymptotic to one of the terms in (\ref{middleExpansion}). Referring to Lemma \ref{fixedAsymps}, the asymptotics of each integral term depend on whether $m-s+1-\alpha$ is greater than or less than $-1$. The case when $m-s+1-\alpha = -1$ does not arise since $m$ and $s$ are both integers, and $\alpha$ is not.

We note that since $m<\alpha$ then $(m-\alpha)-s+1 < \alpha - 1$ for all $s \geq 2$, with $s=1$ being a special case. So, using Lemma \ref{fixedAsymps},
\begin{align*}
    \mu_1 n^{-m} \prod_{i=1}^{m} h_i \int_{\frac{\mu n}{h_{0}}}^{\frac{\mu n}{h_1}} h_j^{m} \; dF(h_j) &\asymp \begin{cases} n^{1-\alpha} h_1^{-(m-\alpha)} l \left( \frac{\mu n}{h_1} \right) \prod_{i=2}^{m} h_i & \text{when } \alpha - 1 < m < \alpha\,,\\
     n^{\frac{1}{2}(1-m-\alpha)} l(\sqrt{\mu n}) \prod_{i=1}^{m} h_i   & \text{when } m < \alpha - 1\,, \\
    \end{cases}
    \end{align*}
and
\begin{align*}    
    \mu_s n^{-(m-s+1)} \prod_{i=s}^{m} h_i \int_{\frac{\mu n}{h_{s-1}}}^{\frac{\mu n}{h_s}} h_j^{m-s+1} \; dF(h_j) &\asymp  
     n^{1-\alpha} h_{s-1}^{-(m-s+2-\alpha)} l \left( \frac{\mu n}{h_{s-1}} \right) \prod_{i=s}^{m} h_i\,, \\
    \overline{F}\left( \frac{\mu n}{h_m} \right) &= n^{1-\alpha}h_m^{-(1-\alpha)} l\left( \frac{\mu n}{h_m} \right)\,.
\end{align*}
The result follows after some manipulation of indices.
\end{proof}

We now make use of Lemma \ref{convexLemma} and the linearity of $\overline{J}_{1,m}$ to see that
\begin{align}
    &\overline{J}_{1,m} \left( \left( \int_{\sqrt{\mu n}}^{\infty} \prod_{i=1}^m p_{i,j}\; dF(h_{j}) \right)^v \right) \nonumber \\ 
   & \lesssim \overline{J}_{1,m} \left( C_0^v  n^{\frac{v}{2}(1-m-\alpha)} l(\sqrt{\mu n})^v \prod_{i=1}^{m} h_i^v + \sum_{s=1}^{m} C_s^v n^{v(1-\alpha)} h_s^{-v(m-s+1-\alpha)} l \left( \frac{\mu n}{h_s} \right)^v \prod_{i=s+1}^{m} h_i^v \right) \nonumber \\
   & \lesssim  C_0^v  n^{\frac{v}{2}(1-m-\alpha)} l(\sqrt{\mu n})^v \overline{J}_{1,m} \left(\prod_{i=1}^{m} h_i^v \right)\nonumber\\
    &\qquad + \sum_{s=1}^{m} C_s^v n^{v(1-\alpha)} \overline{J}_{1,m} \left( h_s^{-v(m-s+1-\alpha)} l \left( \frac{\mu n}{h_s} \right)^v \prod_{i=s+1}^{m} h_i^v \right)\,. \label{fullExpandedSum}
\end{align}

It remains to show that, depending on $k$, each of the terms of (\ref{fullExpandedSum}) is asymptotically bounded above by the desired function. We do this via four lemmas.

\begin{lemma} \label{JLemo1}
For integers $m,v\geq1$ such that $m+v = k > \alpha$ and an integer $r$ such that $0 \leq r \leq m-1$,
\begin{align*}
    \overline{J}_{m-r,m} \left(\prod_{i=1}^{m} h_i^v \right) \lesssim n^{\frac{r+1}{2}(m+v-\alpha)} l(\sqrt{n})^{r+1} \prod_{i=1}^{m-r-1}h_i^v\,. 
\end{align*}
\end{lemma}
\begin{proof}
We proceed by induction on $r$. Note that
\begin{equation*}
    \overline{J}_{m,m} \left(\prod_{i=1}^{m} h_i^v \right) = J_{m,m} \left(\prod_{i=1}^{m} h_i^v \right)
    \asymp h_{m-1}^{m+v-\alpha} l(h_{m-1}) \prod_{i=1}^{m-1} h_i^v \lesssim n^{\frac{1}{2}(m+v-\alpha)} l \left( \sqrt{n} \right) \prod_{i=1}^{m-1} h_i^v \,, 
\end{equation*}
where we used Corollary \ref{JCor}, Lemma \ref{escapeLemma} and the fact that $m+v = k > \alpha$. Similarly, we have
\begin{align*}
    \overline{J}_{m-r,m} \left(\prod_{i=1}^{m} h_i^v \right) &\lesssim J_{m-r,m} \left(n^{\frac{r}{2}(m+v-\alpha)} l(\sqrt{n})^{r} \prod_{i=1}^{m-r}h_i^v \right) \\
    &\asymp n^{\frac{r}{2}(m+v-\alpha)} l(\sqrt{n})^{r} h_{m-r-1}^{m+v-\alpha} l \left( h_{m-r-1} \right) \prod_{i=1}^{m-r-1}h_i^v \\
    &\lesssim n^{\frac{r+1}{2}(m+v-\alpha)} l(\sqrt{n})^{r+1} \prod_{i=1}^{m-r-1}h_i^v\,,
\end{align*}
which completes the proof.
\end{proof}

\begin{lemma} \label{JLemo2}
For integers $m,v\geq1$ such that $m+v = k > \alpha$ and an integer $s$ such that $1 \leq s \leq m$, let $\sigma$ be such that $m+v-\alpha > \sigma > 0$. Then, for an integer $r$ such that $0 \leq r \leq m-1$,
\begin{align*}
    &\overline{J}_{m-r,m} \left( h_s^{-v(m-s+1-\alpha)} l \left( \frac{\mu n}{h_s} \right)^v \prod_{i=s+1}^{m} h_i^v \right)  \\ &\lesssim \begin{cases} n^{\frac{r+1}{2}(m+v-\alpha)} h_s^{-v(m-s+1-\alpha)} l \left( \frac{\mu n}{h_s} \right)^v l \left( \sqrt{n} \right)^{r+1} \prod_{i=s+1}^{m-r-1} h_i^v & \text{when } s < m-r \leq m\,, \\
    n^{\frac{r}{2}(m+v-\alpha)+\frac{1}{2}(m-\alpha)(1-v)+\frac{\sigma}{2}} l \left( \sqrt{n} \right)^{v+r+1} h_{m-r-1}^{v(m-r-1)-\sigma} & \text{when } 1\leq m-r \leq s\,.
    \end{cases}
\end{align*}
\end{lemma}

\begin{proof}

Again we proceed by induction on $r$. If $s \neq m$ we use the same methodology as in Lemma \ref{JLemo1} to see
\begin{align*}
    \overline{J}_{m,m} \left( h_s^{-v(m-s+1-\alpha)} l \left( \frac{\mu n}{h_s} \right)^v \prod_{i=s+1}^{m} h_i^v \right) &\lesssim n^{\frac{1}{2}(m+v-\alpha)} l\left( \sqrt{n} \right) h_s^{-v(m-s+1-\alpha)} l \left( \frac{\mu n}{h_s} \right)^v \prod_{i=s+1}^{m-1} h_i^v\,,
\end{align*}

Similarly, for $r$ such that $s < m-r \leq m-1$ we can write
\begin{align*}
     &\overline{J}_{m-r,m} \left( h_s^{-v(m-s+1-\alpha)} l \left( \frac{\mu n}{h_s} \right)^v \prod_{i=s+1}^{m} h_i^v \right) \\ 
     &\lesssim J_{m-r,m} \left( n^{\frac{r}{2}(m+v-\alpha)} h_s^{-v(m-s+1-\alpha)} l \left( \frac{\mu n}{h_s} \right)^v l \left( \sqrt{n} \right)^{r} \prod_{i=s+1}^{m-r} h_i^v \right) \\
     &\lesssim n^{\frac{r+1}{2}(m+v-\alpha)} h_s^{-v(m-s+1-\alpha)} l \left( \frac{\mu n}{h_s} \right)^v  l \left( \sqrt{n} \right)^{r+1} \prod_{i=s+1}^{m-r-1} h_i^v\,.
\end{align*}
Hence the lemma holds for all $r$ such that $s < m-r \leq m$. 

We now note that if $s > 1$ then $v(s-1)>0$, so we can choose a constant $\gamma$ such that
\begin{equation}
    \min\left\{v(s-1),\sigma \right\} > v\gamma > 0\,. \label{gammaDef}
\end{equation}
We will address the case where $s=1$ later in this proof.

Note that
\begin{align*}
    -v(m-s+1-\alpha+\gamma)+m-\alpha &= (m-\alpha)(1-v) + v(s-1) - v\gamma >0\,. 
\end{align*}
Since also $(m-\alpha)(1-v)+\sigma-v\gamma > 0$ and $\gamma > 0$, we may use Corollary \ref{JCor}, both parts of Lemma \ref{escapeLemma} and the already proven parts of this lemma to see that
\begin{align*}
    &\overline{J}_{s,m} \left( h_s^{-v(m-s+1-\alpha)} l \left( \frac{\mu n}{h_s} \right)^v \prod_{i=s+1}^{m} h_i^v \right) \\ & \lesssim J_{s,m} \left( n^{\frac{m-s}{2}(m+v-\alpha)} h_s^{-v(m-s+1-\alpha)} l \left( \frac{\mu n}{h_s} \right)^v  l \left( \sqrt{n} \right)^{m-s} \right) \\
   & \lesssim n^{\frac{v\gamma}{2}+\frac{m-s}{2}(m+v-\alpha)} h_{s-1}^{(m-\alpha)(1-v)+v(s-1)-v\gamma} l \left( h_{s-1} \right) l \left( \sqrt{n} \right)^{m+v-s} \\
   & \lesssim n^{\frac{m-s}{2}(m+v-\alpha)+\frac{1}{2}(m-\alpha)(1-v) +\frac{\sigma}{2}} h_{s-1}^{v(s-1)-\sigma} l \left( \sqrt{n} \right)^{m+v-s+1}\,.
\end{align*}
Note that the first line here (and hence all following lines) is consistent with the case where $s=m$.

Observe that for $r$ such that $m-r > 1$ we have
\begin{align*}
    v(m-r-1)-\sigma + m - \alpha &= ((m+v-\alpha)-\sigma)+v(m-r-2) > 0\,.
\end{align*}
So, for $r$ such that $1\leq m-r<s$ we see that
\begin{align*}
    &\overline{J}_{m-r,m} \left( h_s^{-v(m-s+1-\alpha)} l \left( \frac{\mu n}{h_s} \right)^v \prod_{i=s+1}^{m} h_i^v \right) \\ & \lesssim J_{m-r,m} \left( n^{\frac{r-1}{2}(m+v-\alpha)+\frac{1}{2}(m-\alpha)(1-v)+\frac{\sigma}{2}} l \left( \sqrt{n} \right)^{v+r} h_{m-r}^{v(m-r)-\sigma} \right) \\
   & \lesssim n^{\frac{r-1}{2}(m+v-\alpha)+\frac{1}{2}(m-\alpha)(1-v)+\frac{\sigma}{2}} l \left( \sqrt{n} \right)^{v+r} h_{m-r-1}^{m+v-\alpha +v(m-r-1)-\sigma} l \left( h_{m-r-1} \right) \\
   & \lesssim n^{\frac{r}{2}(m+v-\alpha)+\frac{1}{2}(m-\alpha)(1-v)+\frac{\sigma}{2}} l \left( \sqrt{n} \right)^{v+r+1} h_{m-r-1}^{v(m-r-1)-\sigma}\,,
\end{align*}
where in the last step we again used Lemma \ref{escapeLemma}. Hence, the lemma holds for $1\leq m-r\leq s$ in the case $s>1$.

In the case where $s=1$, noting that if $l(x)$ is slowly-varying then so is $l(x)^v$, we observe from Potter's Theorem \cite[Theorem 1.5.6]{bingham_goldie_teugels_1987} that there exist constants $D_1 ,D_2 >1$ dependent on $\sigma$ such that
\begin{equation*}
    \frac{l \left( \mu n/h_1 \right)^v}{l \left(\sqrt{\mu n} \right)^v} \leq D_1 \left( \frac{\sqrt{\mu n}}{h_1} \right)^{\frac{\sigma}{2}}\qquad\mbox{ and }\qquad
    \frac{l \left( h_1 \right)}{l \left(\sqrt{\mu n} \right)} \leq D_2 \left( \frac{\sqrt{\mu n}}{h_1} \right)^{\frac{\sigma}{2}}\,.
\end{equation*}
Hence, since $l(\sqrt{\mu n})>0$ for all $n\geq1$,
\begin{align}
    l \left( \frac{\mu n}{h_1}\right)^v l \left( h_1 \right) \leq D_1 D_2 (\mu n)^{\frac{\sigma}{2}} h_1^{-\sigma} l \left( \sqrt{\mu n} \right)^{v+1} \asymp n^{\frac{\sigma}{2}} h_1^{-\sigma} l \left( \sqrt{n} \right)^{v+1}\,.
\end{align}
So we see that
\begin{align*}
    &\overline{J}_{s,m} \left( h_s^{-v(m-s+1-\alpha)} l \left( \frac{\mu n}{h_s} \right)^v \prod_{i=s+1}^{m} h_i^v \right) \\ & \lesssim J_{1,m} \left( J_{2,m} \left( n^{\frac{m-2}{2}(m+v-\alpha)} h_1^{-(m-\alpha)} l \left( \frac{\mu n}{h_1} \right)^v l \left( \sqrt{n} \right)^{m-2} h_2^v \right) \right) \\
   & \lesssim J_{1,m} \left( n^{\frac{m-2}{2}(m+v-\alpha)} h_1^{-(m-\alpha)} l \left( \frac{\mu n}{h_1} \right)^v l \left( \sqrt{n} \right)^{m-2} h_1^{m+v-\alpha} l \left( h_1 \right) \right) \\
   & \lesssim J_{1,m} \left( n^{\frac{m-2}{2}(m+v-\alpha)+\frac{\sigma}{2}} l \left( \sqrt{n} \right)^{m+v-1} h_1^{v-\sigma} \right) \\
   & =  n^{\frac{m-1}{2}(m+v-\alpha)+\frac{\sigma}{2}} l \left( \sqrt{n} \right)^{m+v} h_0^{-\sigma}\,,
\end{align*}
and our lemma holds in this case also.
\end{proof}

\begin{lemma} \label{JLemo3}
For integers $m,v\geq1$ such that $m+v = k < \alpha$ and an integer $r$ such that $0 \leq r \leq m-1$,
\begin{align*}
    \overline{J}_{m-r,m} \left(\prod_{i=1}^{m} h_i^v \right) \lesssim \prod_{i=1}^{m-r-1} h_i^v\,.
\end{align*}
\end{lemma}

\begin{proof}
Again, we proceed by induction on $r$:
\begin{equation*}
    \overline{J}_{m,m} \left(\prod_{i=1}^{m} h_i^v \right) = J_{m,m} \left(\prod_{i=1}^{m} h_i^v \right)
    \asymp \prod_{i=1}^{m-1} h_i^v\,,
\end{equation*}
where we used Corollary \ref{JCor} and the fact that $m+v = k < \alpha$. Similarly, we see that for $1\leq r \leq m-1$
\begin{equation*}
    \overline{J}_{m-r,m} \left(\prod_{i=1}^{m} h_i^v \right) \asymp J_{m-r,m} \left(\prod_{i=1}^{m-r} h_i^v \right) 
    \asymp \prod_{i=1}^{m-r-1} h_i^v\,,
\end{equation*}
as required.
\end{proof}

\begin{lemma} \label{JLemo4}
For integers $m,v\geq1$ such that $m+v = k < \alpha$ and an integer $s$ such that $1 \leq s \leq m$, let $\sigma_1,\sigma_2,\ldots, \sigma_s$ satisfy $\sigma_s > \sigma_{s-1} > \ldots > \sigma_2 > \sigma_1 > 0$ and
\begin{align*}
    2v(\alpha-m-v)+v(v-1) > \sigma_s &> \max \left\{ 0, (s-v)(m-\alpha)+v(s-1)\right\}\,.
\end{align*}
Then, for an integer $r$ such that $0 \leq r \leq m-1$,
\begin{align*}
    &\overline{J}_{m-r,m} \left( h_s^{-v(m-s+1-\alpha)} l \left( \frac{\mu n}{h_s} \right)^v \prod_{i=s+1}^{m} h_i^v \right)  \\ &\lesssim \begin{cases} 
     h_s^{-v(m-s+1-\alpha)} l \left( \frac{\mu n}{h_s} \right)^v \prod_{i=s+1}^{m-r-1} h_i^v & \text{when } s< m-r \leq m\,,  \\
    n^{\frac{\sigma_{q+2}}{2}} l \left( \sqrt{n} \right)^{v+q+1} h_{m-r-1}^{((q+1)-v)(m-\alpha)+v(s-1) - \sigma_{q+2}} & \text{when } s-t'+1 < m-r \leq s\,, \\
    n^{\frac{\sigma_{t'}}{2}} l \left( \sqrt{n} \right)^{v+t'-1} & \text{when } 1 \leq m-r \leq s-t'+1\,,
    \end{cases}
\end{align*}
where $q = s - m + r$, $\eta_t = (t-v)(m-\alpha)+v(s-1)-\sigma_t$ and $t' = \min\left\{ t\ |\ 1 \leq t \leq s, \eta_t < 0 \right\}$.
\end{lemma}

\begin{proof}
Since $v \geq 1$ and $\alpha > k = m + v$, we have $2v(\alpha-m-v)+v(v-1) > 0$. We also note that
\begin{align*}
    2v(\alpha-m-v)+v(v-1) - (s-v)(m-\alpha)+v(s-1) &= (v+s)(\alpha - m - v) > 0\,,
\end{align*}
and so the interval $\left(\max \left\{0, (s-v)(m-\alpha) + v(s-1) \right\},2v(\alpha-m-v)+v(v-1)\right)$ is not empty. Hence $\sigma_s$ (and by extension $\sigma_i$ for $1 \leq i \leq s-1$) is well-defined.

Noting that $\eta_{t+1}-\eta_t = (m-\alpha)+(\sigma_{t}-\sigma_{t+1})< 0$, we see that the $\eta_t$ form a decreasing sequence. We also see, by the definition of $\sigma_s$ that
\begin{align*}
    \eta_s = (s-v)(m-\alpha) + v(s-1) - \sigma_s <0\,.
\end{align*}
Hence the set $\left\{ t | 1 \leq t \leq s, \eta_t < 0 \right\}$ is non-empty, and $t' \leq s$. Since the $\sigma_i$ are arbitrary within an interval, we can make adjustments to them so that there does not exist a $t$ such that $\eta_t = 0$. 

We proceed again by induction on $r$. Using Corollary \ref{JCor} we see that, in the case $s\neq m$,
\begin{align*}
    \overline{J}_{m,m} \left( h_s^{-v(m-s+1-\alpha)} l \left( \frac{\mu n}{h_s} \right)^v \prod_{i=s+1}^{m} h_i^v \right) &= J_{m,m} \left( h_s^{-v(m-s+1-\alpha)} l \left( \frac{\mu n}{h_s} \right)^v \prod_{i=s+1}^{m} h_i^v \right)  \\
    &\asymp h_s^{-v(m-s+1-\alpha)} l \left( \frac{\mu n}{h_s} \right)^v \prod_{i=s+1}^{m-1} h_i^v\,.
\end{align*}
Similarly, for $r$ such that $s< m-r \leq m-1$ we have
\begin{align*}
    \overline{J}_{m-r,m} \left( h_s^{-v(m-s+1-\alpha)} l \left( \frac{\mu n}{h_s} \right)^v \prod_{i=s+1}^{m} h_i^v \right) &\asymp J_{m-r.m} \left( h_s^{-v(m-s+1-\alpha)} l \left( \frac{\mu n}{h_s} \right)^v \prod_{i=s+1}^{m-r} h_i^v \right) \\
    &\asymp h_s^{-v(m-s+1-\alpha)} l \left( \frac{\mu n}{h_s} \right)^v \prod_{i=s+1}^{m-r-1} h_i^v\,.
\end{align*}
Hence, the lemma holds for all $r$ such that $s < m-r \leq m$.

Define $\gamma = \frac{\sigma_1}{v} > 0 $. Note that
\begin{align*}
    -v(m-s+1-\alpha+\gamma) + m -\alpha &= (1-v)(m-\alpha)+v(s-1) - \sigma_1 \\
    &= \eta_1 \begin{cases} < 0 & \text{when } t' = 1\,, \\ > 0 & \text{when } t' \neq 1\,, \end{cases}
\end{align*}
so we can continue to use Corollary \ref{JCor} alongside Lemma \ref{escapeLemma} to see that
\begin{align*}
    &\overline{J}_{s,m} \left( h_s^{-v(m-s+1-\alpha)} l \left( \frac{\mu n}{h_s} \right)^v \prod_{i=s+1}^{m} h_i^v \right) \\ & \lesssim J_{s,m} \left( h_s^{-v(m-s+1-\alpha)} l \left( \frac{\mu n}{h_s} \right)^v \right) \\
   & \lesssim J_{s,m} \left( n^{\frac{v\gamma}{2}}h_s^{-v(m-s+1-\alpha+\gamma)} l \left( \sqrt{n} \right)^v\right) \\
   & \lesssim \begin{cases} n^{\frac{\sigma_1}{2}}l \left( \sqrt{n} \right)^v & \text{when } t' = 1\,, \\  n^{\frac{\sigma_1}{2}}h_{s-1}^{(1-v)(m-\alpha)+v(s-1)-\sigma_{1}} l \left( h_{s-1} \right) l \left( \sqrt{n} \right)^v & \text{when } t' \neq 1\,.
    \end{cases}
\end{align*}
The second line here is consistent with the case where $s=m$. Recall that $\sigma_{i+1}-\sigma_{i} >0$
so, in the case $t' \neq 1$, we may use Lemma \ref{escapeLemma} to see
\begin{align*}
    \overline{J}_{s,m} \left( h_s^{-v(m-s+1-\alpha)} l \left( \frac{\mu n}{h_s} \right)^v \prod_{i=s+1}^{m} h_i^v \right) 
    & \lesssim n^{\frac{\sigma_1}{2}}h_{s-1}^{(1-v)(m-\alpha)+v(s-1)-\sigma_{1}} l \left( h_{s-1} \right) l \left( \sqrt{n} \right)^v \\
   & \lesssim n^{\frac{\sigma_2}{2}}h_{s-1}^{(1-v) (m-\alpha)+v(s-1)-\sigma_{2}} l \left( \sqrt{n} \right)^{v+1}\,.
\end{align*}

Let $q = s-m+r$. For $r$ such that $s - t' + 1 < m-r \leq s$ we have $ 0 \leq q+1 < t'$. So we see that
\begin{align*}
    (q-v)(m-\alpha)+v(s-1) - \sigma_{q+1} + m - \alpha &= ((q+1)-v)(m-\alpha)+v(s-1) - \sigma_{q+1} \\
    &= \eta_{q+1} > 0\,.
\end{align*}
So, remaining in this case, and continuing to use Corollary \ref{JCor} and Lemma \ref{escapeLemma} we have
\begin{align*}
     &\overline{J}_{m-r,m} \left( h_s^{-v(m-s+1-\alpha)} l \left( \frac{\mu n}{h_s} \right)^v \prod_{i=s+1}^{m} h_i^v \right) \\ & \lesssim J_{m-r,m} \left(  n^{\frac{\sigma_{q+1}}{2}} l \left( \sqrt{n} \right)^{v+q} h_{m-r-1}^{(q-v)(m-\alpha)+v(s-1) - \sigma_{q+1}} \right) \\
    & \lesssim n^{\frac{\sigma_{q+1}}{2}} l \left( \sqrt{n} \right)^{v+q} h_{m-r-1}^{((q+1)-v)(m-\alpha)+v(s-1) - \sigma_{q+1}} l \left( h_{m-r-1}\right) \\
    & \lesssim n^{\frac{\sigma_{q+2}}{2}} l \left( \sqrt{n} \right)^{v+q+1} h_{m-r-1}^{((q+1)-v)(m-\alpha)+v(s-1) - \sigma_{q+2}}\,,
\end{align*}
and thus the lemma holds for all $r$ such that $s-t'+1< m-r \leq s$.

Recall that, by definition, $\eta_{t'} = (t'-v)(m-\alpha)+v(s-1) - \sigma_{t'} < 0$. Hence, noting that if $m-r = s - t' + 1$ then $q = t' - 1$, we see that
\begin{align*}
    &\overline{J}_{s-t'+1,m} \left( h_s^{-v(m-s+1-\alpha)} l \left( \frac{\mu n}{h_s} \right)^v \prod_{i=s+1}^{m} h_i^v \right) \\ & \lesssim J_{s-t'+1,m} \left( n^{\frac{\sigma_{t'}}{2}} l \left( \sqrt{n} \right)^{v+t'-1} h_{m-r-1}^{((t'-1-v)(m-\alpha)+v(s-1) - \sigma_{t'}} \right) \\
   & \lesssim n^{\frac{\sigma_{t'}}{2}} l \left( \sqrt{n} \right)^{v+t'-1}\,.
\end{align*}
Similarly, noting that $0+m-\alpha < 0$ we see that, for $r$ such that $1\leq m-r <s-t'+1$,
\begin{align*}
    \overline{J}_{m-r,m} \left( h_s^{-v(m-s+1-\alpha)} l \left( \frac{\mu n}{h_s} \right)^v \prod_{i=s+1}^{m} h_i^v \right) &\lesssim J_{m-r,m} \left( n^{\frac{\sigma_{t'}}{2}} l \left( \sqrt{n} \right)^{v+t'-1} h_{m-r}^0 \right) \\
    &\lesssim n^{\frac{\sigma_{t'}}{2}} l \left( \sqrt{n} \right)^{v+t'-1}\,.
\end{align*}
So, the lemma holds for all $r$ such that $1 \leq m-r \leq s-t'+1$ and the proof is complete.
\end{proof}

Now, returning to (\ref{fullExpandedSum}), in the case $k > \alpha$ we can use Lemmas \ref{JLemo1} and \ref{JLemo2} to see that
\begin{align*}
    &\overline{J}_{1,m} \left( \left( \int_{\sqrt{n}}^{\infty} \prod_{i=1}^m p_{i,j}\; dF(h_{j}) \right)^v \right) \\ 
    & \lesssim 
   C_0^v  n^{\frac{m+v}{2}(1-\alpha)-\frac{m}{2}(1-m)} l(\sqrt{n})^{m+v} 
    + \sum_{s=1}^{m} C_s^v n^{\frac{m+v}{2}(1-\alpha) - \frac{m}{2}(1-m)} l \left( \sqrt{n} \right)^{m+v} \\
    & \asymp n^{\frac{k}{2}(1-\alpha)-\frac{m}{2}(1-m)} l \left( \sqrt{n} \right)^k\,.
\end{align*}

In the case $k<\alpha$, we note two things. Firstly,
\begin{align*}
    \left( \frac{m+v}{2}(1-m-v) - \frac{m}{2}(1-m) \right) - \frac{v}{2}(1-m-\alpha) = \frac{v}{2}(\alpha - m - v) > 0\,,
\end{align*}
so $ \frac{v}{2}(1-m-\alpha) < \frac{m+v}{2}(1-m-v) - \frac{m}{2}(1-m)$ and, using Lemma \ref{powerDominatesSlowVary} alongside the property that if $l(x)$ is slowly-varying then so is $l(x)^a$ for any exponent $a$, we see that
\begin{align*}
    n^{\frac{v}{2}(1-m-\alpha)} l(\sqrt{\mu n})^v \asymp n^{\frac{v}{2}(1-m-\alpha)} l(\sqrt{n})^v \lesssim n^{\frac{m+v}{2}(1-m-v) - \frac{m}{2}(1-m)}\,.
\end{align*}
Secondly we note, for each fixed $s$, that $\frac{\sigma_{t'}}{2} < \frac{\sigma_s}{2} < v(\alpha - m - v) + \frac{v}{2}(v-1)$, and so
\begin{align*}
    n^{\frac{\sigma_{t'}}{2}}l \left( \sqrt{n} \right)^{v+t'-1} \lesssim n^{v(\alpha - m - v) + \frac{v}{2}(v-1)}\,,
\end{align*}
where we again used Lemma \ref{powerDominatesSlowVary} and the fact that if $l(x)$ is slowly-varying then so is $l(x)^a$. Hence, using Lemmas \ref{JLemo3} and \ref{JLemo4} on (\ref{fullExpandedSum}) we see that
\begin{align*}
    &\overline{J}_{1,m} \left( \left( \int_{\sqrt{n}}^{\infty} \prod_{i=1}^m p_{i,j}\; dF(h_{j}) \right)^v \right) \\ 
    & \lesssim  C_0^v  n^{\frac{v}{2}(1-m-\alpha)} l(\sqrt{\mu n})^v + \sum_{s=1}^{m} C_s^v n^{v(1-\alpha)} n^{\frac{\sigma_{t'}}{2}} l \left( \sqrt{n} \right)^{v+t'-1} \\
    & \lesssim  
    C_0^v  n^{\frac{m+v}{2}(1-m-v)-\frac{m}{2}(1-m)} + \sum_{s=1}^{m} C_s^v n^{\frac{m+v}{2}(1-m-v)-\frac{m}{2}(1-m)} \\
    & \asymp n^{\frac{k}{2}(1-k)-\frac{m}{2}(1-m)}\,.
\end{align*}

We see finally that
\begin{align*}
    I_m & \asymp n^{\frac{m}{2}(1-m)} \overline{J}_{1,m} \left( \left( \int_{\sqrt{n}}^{\infty} \prod_{i=1}^m p_{i,j}\; dF(h_{j}) \right)^v \right) \\
    & \lesssim \begin{cases} n^{\frac{k}{2}(1-k)} & \text{when } k < \alpha\,, \\ n^{\frac{k}{2}(1-\alpha)} l \left( \sqrt{n} \right)^k & \text{when } k > \alpha\,, \end{cases}
\end{align*}
which completes the proof of Theorem \ref{alphaNonInt} when combined with (\ref{eq:cond}) and our results on the extreme cases obtained above.

\section{Results for integer $\alpha$} \label{alphaIntSec}

In the case where $\alpha$ is an integer, the techniques used above in proving Theorem \ref{alphaNonInt} do not yield sharp asymptotics when $k \geq \alpha$; see, for example, the proof of Lemma \ref{middleExpansionLem} for one step of the proof which breaks down in the case where $\alpha$ is an integer.  We do, however, have asymptotic upper bounds:

\begin{theorem} \label{alphaInt}
 For $\alpha \in (2,\infty)\cap \mathbb{Z}$ and $k \geq 2$,
 \begin{align}
\mathbb{P} \left( K_k \right) \begin{cases} \asymp n^{\frac{k}{2}(1-k)} & \text{ when } k<\alpha\,, \\
    \lesssim n^{\frac{k}{2}(1-k)}Q(\sqrt{n})^{k-1} Q(n)& \text{ when } k=\alpha\,, \\
    \lesssim n^{\frac{k}{2}(1-\alpha)} l \left( \sqrt{n} \right)^{\alpha-1} Q(n)^{k-\alpha+1} & \text{ when } k>\alpha\,,
    \end{cases}
\end{align}
where
\begin{align}
    Q(x) = \int_1^{x} h^{\alpha-1}\ dF(h) 
    \label{QDef}
\end{align}
for $x\geq1$.
\end{theorem}

Notably, the asymptotics seen in Theorem \ref{alphaNonInt} are still lower bounds for the asymptotics of $\mathbb{P}\left( K_k \right)$. The corresponding lower bound for the $k=\alpha$ case (found from the extreme case where every node has weight less than $\sqrt{\mu n}$) is $n^{\frac{k}{2}(1-k)}Q(\sqrt{n})^{k}$.

We note the following useful properties of $Q(x)$:
\begin{enumerate}
    \item $Q(x)$ is positive and non-decreasing,
    \item $Q(x)$ is slowly-varying, and
    \item $l(x) \lesssim Q(x)$.
\end{enumerate}
Closed forms for the asymptotics of $Q(x)$ for a wide class of slowly-varying functions $l(h)$ can be derived from Polfeldt's results in \cite{QAsymps}, but a general form for all $l(h)$ is yet to be found.

The full proof of Theorem \ref{alphaInt} is omitted here for reasons of brevity, but we shall outline key steps in its proof. 

A careful analysis of the proof of Theorem \ref{alphaNonInt} reveals that the only point where we used the assumption that $\alpha$ was non-integer (other than the assumption that $k \neq \alpha$) was in Lemma \ref{middleExpansionLem}, where we used it to assert that $m - s + 1 - \alpha \neq - 1$. Note that if both $m$ and $s$ are integers, with $1 \leq m \leq \alpha$ and $1 \leq s \leq m$, $m - s + 1 - \alpha = - 1$ only in the cases $m = \alpha$ ($s = 2$) and $m = \alpha - 1$ ($s = 1$). Neither of these cases can arise when $k < \alpha$, so there the proof of Theorem \ref{alphaNonInt} applies and we obtain a sharp asymptotic for $\mathbb{P}\left( K_k \right)$.

In the remaining cases we may use the fact that $Q(x)$ is positive and non-decreasing to write
\begin{equation*}
    \int_{\frac{\mu n}{h_{s-1}}}^{\frac{\mu n}{h_s}} h_j^{m-s+1}\; dF(h_j) = Q \left( \frac{\mu n}{h_s} \right) - Q \left( \frac{\mu n}{h_{s-1}} \right) 
    \leq Q \left( \frac{\mu n}{h_s} \right) 
    \leq Q(\mu n) 
    \asymp Q(n)\,.
\end{equation*}
From this relation the asymptotic upper bounds seen in Theorem \ref{alphaInt} arise. Note that all the remaining terms (including in the omitted case $k=\alpha$) are asymptotically bounded above by the extreme case where every node has weight at most $\sqrt{\mu n}$. 

Setting $l(x)$ to be a specific function sheds more light on the situation. In both examples below it is assumed that $\alpha$ is some fixed integer.

When $l(x) = 1$, we have $Q(x) = (\alpha - 1) \log(x) \asymp \log(x)$ and
    \begin{align*}
        \mathbb{P} \left( K_k \right) \asymp \begin{cases} n^{\frac{k}{2}(1-k)} & \text{when } k < \alpha\,, \\ n^{\frac{k}{2}(1-k)}\log ( \sqrt{n} )^{k} & \text{when } k = \alpha\,, \\ n^{\frac{k}{2}(1-\alpha)} & \text{when } k > \alpha\,. \end{cases}
    \end{align*}
    
When $l(x) = \log(x)$, we have $Q(x) = \frac{\alpha - 1}{2} \log(x)^2 - \log(x) \asymp \log(x)^2$ and
    \begin{align*}
        \mathbb{P} \left( K_k \right)  \begin{cases} \asymp n^{\frac{k}{2}(1-k)} & \text{when } k < \alpha\,, \\  \asymp n^{\frac{k}{2}(1-k)} \log(\sqrt{n})^{2k} & \text{when } k = \alpha\,, \\ \lesssim n^{\frac{k}{2}(1-\alpha)} \log(\sqrt{n})^{2k-\alpha-1} & \text{when } k > \alpha\,. \end{cases}
    \end{align*}
    
The case where $k > \alpha$ and $l(x) = \log (x)$ remains as an upper bound since it is unproven whether that bound is attained in all cases. However, using computer methods to fully expand the cumbersome terms appearing in calculations it can be shown that it is attained for certain (relatively) small values of $k$ and $\alpha$.

These examples show that the upper bound presented in Theorem \ref{alphaInt} is not always attained, but asymptotics for $\mathbb{P}\left( K_k \right)$ when $\alpha$ is integer may exceed the corresponding asymptotics in the case where $\alpha$ is non-integer, though they are not guaranteed to do so.

\appendix
\section{Proofs of lemmas in Section \ref{sec:lemmas}} \label{InitProofs}
\begin{proof}[Proof of Lemma \ref{fixedAsymps}]
We first note that
\begin{align*}
    \frac{dF(h)}{dh} = \frac{d}{dh} (1-\overline{F}(h)) = \frac{d}{dh} (1-h^{1-\alpha}l(h)) = \frac{d}{dh} (-h^{1-\alpha}l(h))\,.
\end{align*}
So, starting with (\ref{fixedBottom}), we can integrate by parts to see that 
\begin{align*}
     \int_X^{x(n)} h^\beta\ dF(h) = X^{\beta-\alpha+1}l(X) - x^{\beta-\alpha+1}l(x) + \beta \int_X^{x(n)} h^{\beta-\alpha}l(h)\ dh\,.
\end{align*}
In the case that $\beta-\alpha > -1$ we use \cite[Proposition 1.5.8]{bingham_goldie_teugels_1987} to see that the integral on the right-hand side is asymptotically equivalent to
\begin{align*}
     \beta \int_X^{x(n)} h^{\beta-\alpha}l(h)\ dh \sim \frac{\beta}{\beta-\alpha+1} x^{\beta-\alpha+1} l(x)\,,
\end{align*}
and using \cite[Theorem 1.5.4]{bingham_goldie_teugels_1987} we see that $x^{\beta-\alpha+1} l(x)$ is asymptotic to a non-decreasing function, so it dominates the constant term. Hence the first part of (\ref{fixedBottom}) holds.

In the case $\beta-\alpha < -1$, we first observe, again by \cite[Theorem 1.5.4]{bingham_goldie_teugels_1987}, that $x^{\beta-\alpha+1} l(x)$ is asymptotic to a non-increasing function, and hence is dominated by the constant term $X^{\beta-\alpha+1}l(X)$. For the integral term we write
\begin{align*}
     \beta \int_X^{x(n)} h^{\beta-\alpha}l(h)\ dh = \beta \int_X^{\infty} h^{\beta-\alpha}l(h)\ dh - \beta \int_x^{\infty} h^{\beta-\alpha}l(h)\ dh\,.
\end{align*}
From \cite[Proposition 1.5.10]{bingham_goldie_teugels_1987} we see that the first term here converges (and hence is a well-defined constant) and the second term is asymptotically equivalent to $x^{\beta-\alpha+1}l(x)$, and hence dominated by the constant. Hence the second part of (\ref{fixedBottom}) holds.

For (\ref{fixedNone}), integrating by parts gives
\begin{align*}
    \int_{x(n)}^{y(n)} h^\beta\ dF(h) = y^{\beta-\alpha+1}l(y) - x^{\beta-\alpha+1}l(x) + \beta \int_{x}^{y}h^{\beta-\alpha}l(h)\ dh\,.
\end{align*}
By \cite[Theorem 1.5.4]{bingham_goldie_teugels_1987}, if $\beta-\alpha>-1$ then $\phi(h) = h^{\beta-\alpha+1}l(h)$ is asymptotic to a non-decreasing function, so $\phi(y)$ dominates $\phi(x)$. By the same theorem, if $\beta-\alpha<-1$ then $\phi(h)$ is asymptotic to a non-increasing function, so $\phi(x)$ dominates $\phi(y)$. 

We now consider the integral term. If $\beta-\alpha>-1$, then we write
\begin{align*}
    \int_{x}^{y}h^{\beta-\alpha}l(h)\ dh =  \int_{h_0}^{y}h^{\beta-\alpha}l(h)\ dh -  \int_{h_0}^{x}h^{\beta-\alpha}l(h)\ dh\,.
\end{align*}
From this we use \cite[Proposition 1.5.8]{bingham_goldie_teugels_1987} to see that the first term is asymptotic to $\phi(y)$ and the second is asymptotic to $\phi(x)$ (which is dominated by $\phi(y)$). Hence the first part of (\ref{fixedNone}) holds.

Similarly, if $\beta-\alpha<-1$, then we write
\begin{align*}
     \int_{x}^{y}h^{\beta-\alpha}l(h)\ dh =  \int_{x}^{\infty}h^{\beta-\alpha}l(h)\ dh -  \int_{y}^{\infty}h^{\beta-\alpha}l(h)\ dh\,.
\end{align*}
From this we use \cite[Proposition 1.5.10]{bingham_goldie_teugels_1987} to see that the first term is asymptotic to $\phi(x)$ and the second is asymptotic to $\phi(y)$ (which is dominated by $\phi(x)$). Hence the second part of (\ref{fixedNone}) holds.
\end{proof}

\begin{proof}[Proof of Corollary \ref{JCor}]
Observe that, since both $B$ and $h_i^\beta$ are positive for $h_i \in [1,h_{i-1}]$,
\begin{align*}
    J_{i,m} (g) \begin{Bmatrix} \asymp \\ \text{or} \\ \lesssim \end{Bmatrix} & \int_1^{h_{i-1}} B(n,h_1,h_2,\ldots,h_{i-1})h_i^{\beta+m-1} \ dF(h_i) \\
    =\quad & B(n,h_1,h_2,\ldots,h_{i-1}) \int_1^{h_{i-1}} h_i^{\beta+m-1} \ dF(h_i)\,.
\end{align*}
The result then follows from Lemma \ref{fixedAsymps}.
\end{proof}

\begin{proof}[Proof of Lemma \ref{escapeLemma}]
By \cite[Theorem 1.5.4]{bingham_goldie_teugels_1987} there exists a non-decreasing function $\phi(x)$ and a non-increasing function $\psi(x)$ such that
\begin{align*}
h_i^\gamma l \left( h_i\right) \sim \phi(h_i) \leq \phi (\sqrt{\mu n}) \sim \left(\mu n \right)^{\frac{\gamma}{2}} l \left( \sqrt{\mu n}\right) \asymp n^{\frac{\gamma}{2}} l \left( \sqrt{n}\right)\,,   
\end{align*}
and
\begin{align*}
    \frac{n^{-\gamma}}{h_i^{-\gamma}} l \left( \frac{\mu n}{h_i} \right) \sim \mu^{\gamma} \psi \left( \frac{\mu n}{h_i} \right) \leq \mu^{\gamma}  \psi \left( \sqrt{\mu n} \right) &\sim \mu^{\frac{\gamma}{2}} n^{-\frac{\gamma}{2}} l \left( \sqrt{\mu n} \right)\asymp n^{-\frac{\gamma}{2}} l \left( \sqrt{n} \right)\,.
\end{align*}
\end{proof}

In conclusion, we note that Lemma \ref{powerDominatesSlowVary} is a direct corollary of \cite[Theorem 1.5.4]{bingham_goldie_teugels_1987} and Lemma \ref{convexLemma} follows immediately from the convexity of the function $g(x)=x^v$.

%

\bibliographystyle{abbrv}
\bibliography{biblio}

\end{document}